\documentclass[12pt]{amsart}
\usepackage{amsmath,amssymb,amscd,amsthm,amsxtra}
\usepackage{latexsym}
\usepackage{graphicx}

\usepackage[colorlinks=true, pdfstartview=FitV, linkcolor=blue, citecolor=blue, urlcolor=blue]{hyperref}
\usepackage{color}

\usepackage{graphicx,epstopdf,verbatim}

\allowdisplaybreaks

\newtheorem{theorem}{Theorem}[section]

\newcommand{\RN }{\mathbb R^n}
\newcommand{\R }{\mathbb R}

\newcommand{\supp}{{\rm supp}\,}

\newcommand{\cz}{Calder\'on-Zygmund }

\newcommand{\mMlog}{{\mathcal M}_{L(\log L)}}

\numberwithin{equation}{section}

\let \la=\lambda
\let \e=\varepsilon
\let \d=\delta

\let \f=\varphi

\let \P=\Phi

\newcommand{\Tsigma}{T_{\Sigma\, {\rm \bf b}}}
\newcommand{\Tpi}{T_{\Pi\, {\rm \bf b}}}

\begin{document}

\title [Multilinear commutators]
{Multilinear commutators}

\title [End-point estimates for  iterated commutators]
{End-point estimates for iterated commutators of multilinear
singular integrals }

\authors

\

\author[C. P\'erez]{Carlos P\'erez}
\address{Carlos P\'erez\\
Departamento de An\'alisis Matem\'atico, Facultad de Matem\'aticas,
Universidad de Sevilla, 41080 Sevilla, Spain}
\email{carlosperez@us.es}

\author[G. Pradolini]{Gladis Pradolini}
\address{G. Pradolini\\
Universidad Nacional del Litoral - CONICET\\
G\"uemes 3450, 3000 Santa Fe, Argentina}
\email{gladis.pradolini@gmail.com}

\author[R. H. Torres]{Rodolfo H. Torres}
\address{Rodolfo H. Torres\\
Department of Mathematics, University of Kansas, 405 Snow Hall 1460
Jayhawk Blvd, Lawrence, Kansas 66045-7523, USA }
\email{torres@math.ku.edu}

\author[R. Trujillo-Gonz\'alez]{Rodrigo Trujillo-Gonz\'alez}
\address{Rodrigo Trujillo-Gonz\'alez\\
Departamento de An\'alisis Matem\'atico,  Universidad de La Laguna,
38271 La Laguna, S.C. de Tenerife, Spain } \email{rotrujil@ull.es}


\begin{abstract}

Iterated commutators of multilinear Calder\'on-Zygmund operators and
pointwise multiplication with functions in $BMO$ are studied in
products of Lebesgue spaces. Both strong type and weak end-point
estimates are obtained, including weighted results involving the
vectors weights of the multilinear Calder\'on-Zygmund theory
recently introduced in the literature. Some better than expected
estimates for certain multilinear operators are presented too.

\end{abstract}



\thanks{ The authors  would like to acknowledge  the
support of the following grants. First author and
fourth author: Spanish Ministry of Science and Innovation grant MTM2009-08934 and MTM2008-
05891 respectevely, Second author: Instituto de Matem\'atica Aplicada del Litoral (IMAL-CONICET)-
Universidad Nacional del Litoral, Third author: NSF grant DMS0800492 and a General Research Fund allocation of the University of Kansas. }

\maketitle

\section{Introduction and main results}\label{introduction}

The commutator of a linear \cz operator $T$ and a $BMO$ function
$b$,
$$
T_b(f)=[b,T](f)=bT(f)-T(bf),
$$
was first studied by Coifman, Rochberg, and Weiss \cite{CRW} who
proved that
$$
T_b:L^p(\RN) \to L^p(\RN)
$$
for all $1<p<\infty$. This can be seen as a bilinear result,
$$
BMO(\RN) \times L^p(\RN) \to L^p(\RN),
$$
since actually
\begin{equation}\label{bmolptolp}
\|T_b(f)\|_{L^p} \lesssim \|b\|_{BMO}\, \|f\|_{L^p}.
\end{equation}
Using duality, the above estimate has as an immediate consequence
for $1<p<\infty$ the bilinear estimate
\begin{equation}\label{better}
\|gT(f)-fT^*(g)\|_{H^1}\lesssim \|g\|_{L^{p'}}\, \|f\|_{L^p},
\end{equation}
where $p'$ is the dual exponent of $p$, $H^1$ is the Hardy space,
and $T^*$ is the transpose of $T$.  Note that \eqref{better} is a
{\it better than expected estimate},  since trivially by H\"older's
inequality and the boundedness of $T$,
$$
\|gT(f)-fT^*(g)\|_{L^1}\lesssim \|g\|_{L^{p'}}\, \|f\|_{L^p}.
$$
Both \eqref{bmolptolp} and \eqref{better} put in evidence that some
subtle cancellations are taking place. These estimates have found
many important applications in other areas of operator theory and
partial differential equations.

Another interesting feature of the commutators $T_b$ is the fact
that they fail to satisfy the typical weak end-point $L^1$ estimate
of the \cz theory. As a remedial feature though, they do satisfy
 an alternative $L(\log L)$
end-point estimate, as proved by P\'erez in \cite{P1}  (see \cite{PP}  for a different proof). \vspace{2mm}

Much of the analysis of  linear commutators has been extended to
other context such as weighted spaces, spaces of homogeneous type,
multiparameter and multilinear settings. Higher order and iterated
commutators have been consider too. The literature is by now quite
vast. We will only recount here  the multilinear situation which is
the focus of this article. The purpose of  the present work is to
prove the optimal results for the  iterated commutators  and an
associated multi(sub)linear operator.  In this sense, this article
complements and completes the theory developed by Lerner et al in
\cite{LOPTT}, where the reader will find further bibliography in the
subject. \vspace{2mm}

Let $T$ be an  $m$-linear \cz operator as defined by Grafakos and
Torres in \cite{GT4}  and \cite{GT6} (see the next section for
complete definitions). In particular, such operators satisfy
\begin{equation}\label{standardbound}
T:L^{p_1}(\RN) \times \dots \times L^{p_m}(\RN) \to L^p(\RN)
\end{equation}
whenever $1<p_1,\dots,p_m<\infty$  and
\begin{equation}\label{holderrelation}
\frac1p=\frac{1}{p_1}+\dots+\frac{1}{p_m},
\end{equation}
and also the end-point result
%
%
\begin{equation}\label{weaktype}
T:L^{1}(\RN) \times \dots \times L^{1}(\RN) \to L^{1/m,\infty}(\RN).
\end{equation}

Let  ${\rm \bf b }=(b_1,\dots,b_m)$ be in $BMO^m$.  The commutator
of ${\rm \bf b}$ and the $m$-linear Calder\'on-Zygmund operator $T$,
denoted here by $T_{\Sigma\, {\rm \bf b}}$\footnote{The notation
$T_{\vec b}$ was used instead in \cite{PT} and \cite{LOPTT}. We use
the new notation to better differentiate  this commutator  from the
iterated ones we want to study in this article. The notation for
both types of commutators is also motivated by the estimates they
satisfy.}, was  introduced by P\'erez and Torres in \cite{PT} and is
defined via
\begin{equation}\label{masclara}
T_{\Sigma \, {\rm \bf b}}\,(f_{1},\dots,f_{m})=
\sum_{j=1}^{m}  T_{ b_j}^j(f_{1},\dots,f_{m} ) ,      
\end{equation}
where each term is the commutator of $b_j$ and $T$ in the $j$-th
entry of $T$, that is,
$$
T_{ b_j}^j({\rm \bf f}) \equiv  T_{ b_j}^j(f_{1},\dots,f_{m})\equiv
[b_j,T]_j (f_1,\dots,f_m) $$$$\equiv
b_{j}T(f_{1},\dots,f_{j},\dots,f_{m}) -
T(f_{1},\dots,b_{j}f_{j},\dots,f_{m}).
$$
It was shown in \cite{PT} that $\Tsigma$ satisfies the bounds
\eqref{standardbound} for all indices satisfying
\eqref{holderrelation} with $p>1$. The result was extended  in
\cite{LOPTT} to all $p>1/m$. The estimates are of the form
\begin{equation}\label{commutatorbound}
\|\Tsigma({\rm \bf f})\|_{L^p} \lesssim \left( \sum_{j=1}^m
\|b_j\|_{BMO} \right) \, \prod_{j=1}^m \|f_j\|_{L^{p_j}}
\end{equation}
Moreover,  weighted-$L^p$ versions of the bounds
\eqref{standardbound} were obtained in \cite{LOPTT}  for  weights in
the classes $A_{\vec P}$ (see again the next section for
definitions). These classes of weights introduced in  \cite{LOPTT}
are the largest classes of weights for which all $m$-linear \cz
operators are bounded.

As it may be expected from the situation in the linear case, the
end-point estimate  \eqref{weaktype} does not hold for $\Tsigma$.
Instead the following estimate was also obtained in \cite{LOPTT}
\begin{equation}\label{weaktypetsigma}
\left| \big \{x \in \RN\,:\, |\Tsigma({\rm \bf f})(x)|>t^m \big \}
\right| \leq C({\rm \bf b })\, \prod_{j=1}^m \left(\int_{\RN} \Phi
\left( \dfrac{|f_j(x)|}{t} \right)\, dx\right)^{1/m},
\end{equation}
where  $\Phi(t)= t\,(1+\log^+t)$. The result is still true if the
Lebesgue measured is changed by an $A_{\vec 1}$ weight  \eqref{multiap}.  Note that for $m=1$ this is the
end-point result in \cite{P1}. The estimate \eqref{weaktypetsigma}
is sharp in an appropriate sense. It is also the right one from the
point of view interpolation as recently shown by Grafakos et al
\cite{GLPT}.

The results for  $\Tsigma$ were obtained in \cite{LOPTT} via
corresponding ones for the maximal function
$$  {\mathcal M}_{\Sigma\,L(\log L)} = \sum {\mathcal M}^i_{L(\log L)}, $$
where
$${\mathcal M}^i_{L(\log L)}({\rm \bf f})(x)=\sup_{Q\ni x}  \|f_i\|_{L(\log L),Q} \prod_{j\neq i} \frac{1}{|Q|} \int _{Q} f_j \, dx.
$$
Independently, Tang \cite{tang} has also looked at $\Tsigma$,
iterations of it, and vector valued versions, but only for weights
in the classical $A_p$ classes (whose product is still smaller than
$A_{\vec P}$). He obtained some end-point estimates but with the
right-hand side term in \eqref{weaktypetsigma} replaced by a more
complicated expression with an extra factor, and without the
homogeneity of  \eqref{weaktypetsigma}, which is crucial to obtain
optimality. \vspace{2mm}

We will establish in this article strong bounds for iterated
commutators for $p>1/m$ allowing the full $A_{\vec P}$ classes and
again sharp end-point results when $p=1/m$
%

For a \cz operator $T$ and ${\rm \bf b} = (b_1, \dots, b_m)$ in
$BMO^m$,
 we define the iterated commutators $\Tpi$ to be
\begin{equation}\label{iterado}
\Tpi \,(f_{1},\dots,f_{m}) \equiv
[b_1, [b_2, \dots [b_{m-1},[b_m, T]_m]_{m-1}\dots ]_2]_1({\rm \bf f}).      
\end{equation}

To clarify the notation, if $T$ is associated in the usual way with
a \cz kernel $K$,  then at a  formal level
\begin{equation}\label{formalmente}
\Tpi({\rm \bf f}) (x)= \int_{(\RN)^m} \prod_{j=1}^m
(b_j(x)-b_j(y_j)) \, K(x,y_1,\dots, y_m)f_1(y_1)\dots f_m(y_m)\,
dy_1\dots dy_m.
\end{equation}
%
(See also \eqref{explicito} below for another explicit formula in
the bilinear case.)

We will prove the following strong type bound for $\Tpi$.


%


\begin{theorem}\label{thmStrong}
Let $T$ be an $m$-linear \cz operator;  $\vec w \in A_{\vec P}$ with
$$\frac{1}{p}=\frac{1}{p_1}+\dots+\frac{1}{p_m}$$ and $1<
p_j<\infty,j=1,\dots,m$; and  ${\rm \bf b }\in BMO^m$. Then, there
exists a constant $C$ such that
\begin{equation}\label{daproducto}
\| \Tpi ({\rm \bf f})\|_{L^{p}(\nu_{\vec w})}\le C\,  \prod_{j=1}^m
\|b_j\|_{BMO} \,\prod_{j=1}^m\|f_j\|_{L^{p_j}(w_j)}.
\end{equation}
\end{theorem}

%


%


%

At the end-point we obtain the following estimate.

\begin{theorem}\label{commutatorthmllogl}
Let $T$ be an $m$-linear \cz operator;  $\vec w \in A_{\vec 1}$, and
${\rm \bf b } \in BMO^m$. Then,  there exists a constant $C$
depending on ${\rm \bf b }$ such that
\begin{equation}\label{debilconmuta}
\nu_{\vec w} \left( \big \{x \in \RN\,:\, |\Tpi ({\rm \bf
f})(x)|>t^m \big \} \right) \leq C\, \prod_{j=1}^m \left(\int_{\RN}
\Phi^{(m)} \left(\dfrac{|f_j(x)|}{t} \right)\,
w_j(x)dx\right)^{1/m},
\end{equation}
where $\Phi(t)= t\,(1+\log^+t)$  and
$\Phi^{(m)}=\overbrace{\Phi\circ\cdots\circ\Phi}^{m} $.

\smallskip

Moreover, the estimate is sharp in the sense that $\Phi^{(m)}$ can
not be replaced by $\Phi^{(k)}$ for $k<m$.
\end{theorem}

\vspace{.5cm} \noindent To prove the sharpness of theorem above we
adapt some ideas from \cite{LOPTT}. For simplicity, we consider
$m=2$, $n=1$,  T one of the bilinear operators for $n=1$ obtained from the (linear) Riesz transforms in $n=2$, as it is done for example in \cite{LOPTT}, and the functions
$b_1(x)=b_2(x)=\log|1+x|$ and $f_1=f_2=\chi_{(0,1)}$. We can
prove, for example,  that the estimate
\begin{equation*}
 \left| \big \{x \in \RN\,:\, |\Tpi ({\rm \bf f})(x)|>t^2 \big \}
\right| \leq C\,  \left(\int_{\RN} \Phi \left(\dfrac{|f_1|}{t}
\right)\right)^{1/2} \left(\int_{\RN} \Phi^{(2)}
\left(\dfrac{|f_2|}{t} \right)\right)^{1/2}
\end{equation*}
is false. In fact, if the inequality above were to hold, by the
homogeneity we would have that
\begin{equation*}
 \left| \big \{x \in \RN\,:\, |\Tpi ({\rm \bf f})(x)|>t^2 \big \}
\right|^2 \leq C\,  \int_{\RN} \Phi \left(\dfrac{|f_1|}{t^2} \right)
\int_{\RN} \Phi^{(2)} (f_2),
\end{equation*}
and hence, since $\Phi$ is a Young function
\begin{equation*}
\sup_{\lambda>0} \frac{1}{\Phi (1/\lambda)} \left| \big \{x \in
\RN\,:\, |\Tpi ({\rm \bf f})(x)|>\lambda \big \} \right|^2 <\infty.
\end{equation*}
However, using the fact that $\Phi^{-1}(x)\cong x/(\log x)$ for
$x>e$, it is easy to check that
\begin{eqnarray*}
&&\sup_{\lambda>0} \frac{1}{\Phi (1/\lambda)} \left| \big \{x \in
\RN\,:\, |\Tpi ({\rm \bf f})(x)|>\lambda \big \} \right|^2 \\
&&\leq \sup_{\lambda>0} \frac{1}{\Phi (1/\lambda)} | \big \{x>
e\,:\, \frac{\log^2 (1+x)}{x^2}>\lambda \big \} |^2\\
&&=\sup_{\lambda>0} \frac{1}{\Phi (1/\lambda^2)} | \big \{x>
e\,:\, \frac{\log (x)}{x}>\lambda \big \} |^2\\
&&\ge \sup_{\lambda>0} \frac{1}{\Phi (1/\lambda^2)} | \big \{x>
e\,:\, \frac{C}{\Phi^{-1}(x)}>\lambda \big \} |^2\\
&&\ge C\sup_{0<\lambda< C/e} \frac{(\Phi(C/\lambda)-e)^2}{\Phi
((C/\lambda)^2)} \\
&&\ge \frac{C}{4}\sup_{0<\lambda< C/2e}
\frac{(\Phi(C/\lambda))^2}{\Phi
((C/\lambda)^2)} \\
&&\ge \frac{C}{4}\sup_{0<\lambda< C/2e}\log(C/\lambda)=\infty.
\end{eqnarray*}

\vspace{0.5cm}

As in the linear case and the particular multilinear case studied in
\cite{LOPTT}, the proofs of the two main theorems will be based on
corresponding estimates on a maximal function that controls the
commutator, the operator $\mMlog$ given by
\begin{equation}\label{culyllogl}
{\mathcal M}_{L(\log L)}({\rm \bf f})(x)=\sup_{Q\ni x}\prod_{j=1}^m
\|f_j\|_{L(\log L),Q},
\end{equation}
where the supremum is taken over all cubes $Q$ containing $x$.
Strong bounds for this operator were already obtained in
\cite{LOPTT} but not weak-type ones.  We present in this article the
right end-point distributional estimate it satisfies (see
Theorem~\ref{maximalthmllogl}). This operator  and the estimates it
satisfies are crucial in this paper.
%

Our analysis will show that in fact  one can also study commutators
where only $k<m$ factors appear in \eqref{formalmente}, and  which
are controlled by an appropriate modification of the maximal
function $\mMlog$. We will concentrate only  in the case where there
are $m$ functions in $BMO$, which is the most difficult one, and
leave other generalizations to the interested reader.  See, however Section 3 below.  \vspace{2mm}

The next section contains some basic definitions and further
background related to the classes $A_{\vec P}$  of vector weights
and several multilinear maximal functions from \cite{LOPTT}.
Nevertheless, the reader already familiar with the subject can skip
Section~\ref{preliminares} and move directly to
Section~\ref{sectiontres},  where a key  pointwise estimate
involving the maximal function $\mMlog$, Theorem~
\ref{lemasharpconmuta}, is combined with the classical
Fefferman-Stein inequality to prove the strong bounds in
Theorem~\ref{thmStrong}. Likewise, the proof of
Theorem~\ref{commutatorthmllogl}  is obtained using  a new weak type
estimate for the maximal function $\mMlog$,
Theorem~\ref{maximalthmllogl},  which is presented  in
Section~\ref{seccioncuatro}. \vspace{2mm}

Before we conclude this introduction, we would like to  consider
analogs of \eqref{better} in the multilinear setting in view of
\eqref{daproducto} and put in evidence again some {\it better than
expected estimates}, which are implied by the commutator results and
which also motivate in part our study of commutators. For simplicity
we consider the following particular case.  For a bilinear \cz
operator $T$,  we can write
%
%
\begin{equation}\label{explicito}
\Tpi \, (f_1,f_2) =  b_1b_2 T(f_1,f_2) - b_2T(b_1f_1,f_2)-
b_1T(f_1,b_2f_2) + T(b_1f_1,b_2f_2).
\end{equation}
We can use duality to obtain the surprising quad-linear estimate

\begin{align}
\| hb_2 T(f_1,f_2) - & f_1T^{*1}(h b_2,f_2)-  hT(f_1,b_2f_2) + f_1T^{*1}(h,b_2f_2)\|_{H^1} \nonumber \\
 &\lesssim \|b_2\|_{BMO} \|h\|_{L^{p'}}\|f_1\|_{L^{q}}\|f_2\|_{L^{r}}, \label{multilinearbetter}
\end{align}
%
for $1/q +1/r =1/p$, $1<p,q,r<\infty$, and where $T^{*1}$ is the
transpose of $T$ in the first variable.  Notice that this is again
an improvement (now both in the target and the range) over the
trivial estimate
$$
 S: L^\infty(\RN)  \times L^{p'} (\RN) \times L^q (\RN) \times L^r (\RN) \to L^1(\RN),
$$
where
$$
S(b,h,f_1,f_2)=hb T(f_1,f_2) - f_1T^{*1}(h b,f_2)-  hT(f_1,bf_2) +
f_1T^{*1}(h,bf_2),
$$
and which follows by H\"older's inequality and the boundedness of
$T$. The better estimate obtained reflects again the presence of
certain hidden cancellations. Though we will not carry their study
here any further, it would be interested to see if estimates like
\eqref{multilinearbetter} are amenable to some analysis similar to
the one generated in the linear case as consequence of
\eqref{better}.

{\bf Acknowledgement.}  We would like to thank Hua Wang for pointing out some typos and a gap in a previous version of this manuscript which we have now corrected and filled in.

\section{Some background definitions and estimates}  \label{preliminares}

\subsection{\cz operators} Following \cite{GT4} we will assume here  that $T$ is a bounded $m$-linear \cz operator. That is,
$T$ satifies the bounds \eqref{standardbound} and \eqref{weaktype}
and its Schwartz kernel $K$
 satisfies away from
 the diagonal $x=y_1=\dots=y_m$ in $({\mathbb
R}^n)^{m+1}$,
\begin{equation}\label{cond1} |K(y_0,y_1,\dots,y_m)|\le \frac{A}{
\Bigl(\sum\limits_{k,l=0}^m|y_k-y_l|\Bigr)^{mn}}
\end{equation}
and also
\begin{equation}\label{cond2}
|K(y_0,\dots,y_j,\dots,y_m)-K(y_0,\dots,y_j',\dots,y_m)|
\le\frac{A|y_j-y_j'|^{\e}}{\Bigl(
\sum\limits_{k,l=0}^m|y_k-y_l|\Bigr)^{mn+\e}},
\end{equation}
for some $\e>0$ and all $0\le j\le m$, whenever $|y_j-y_j'|\le
\frac{1}{2}\max_{0\le k\le m}|y_j-y_k|$. In particular for $x \notin
\cap \, \supp f_j$,
$$
T(f_1,\dots,f_m)(x)= \int K(x, y_1,\dots,y_m)f_1(y_1)\dots
f_m(y_m)\,dy_1\dots dy_m.
$$

\bigskip

\subsection{Orlicz norms}\label{orlicz}

%
For  $\Phi(t)= t\,(1+\log^+t)$ and a cube $Q$ in $\RN$  we will
consider the average $\|f\|_{\Phi,Q}$ of a function $f$ given by the
Luxemburg norm
$$
\|f\|_{L\log L, Q}=
\inf \{\la>0\,:\, \frac{1}{|Q|}\int_Q\Phi \left(\frac{|f(x)|}{\la
}\right )dx\leq 1\}.
 $$

We will need the several basic estimates from the theory of Orlicz
spaces. We first recall that
\begin{equation}\label{biggerone}
 \|f\|_{\P, Q} >1 \quad \mbox{if and only if} \quad
\frac{1}{|Q|}\int_Q\P \left(|f(x)|\right )dx > 1.
\end{equation}

Next, we note that the generalized H\"older inequality in Orlicz spaces together with the John-Nirenberg inequality implies that%
\begin{equation}\label{John-Nirenberg-LLogL}
\frac{1}{|Q|}\int_Q \left |b(y)-b_Q \right |f(y)\,dy\leq
C\|b\|_{BMO}\|f\|_{L(\log L),Q},
\end{equation}
an estimate that we shall use in several occasions without further
comment.

We will also use the maximal function
$$ 
M_{L(\log L)}f(x) = \sup_{Q\ni x}\|f\|_{L(\log L), Q},$$
where the supremum is taken over all the cubes containing $x$.
This operator satisfies the pointwise equivalence%
\begin{equation}\label{iteration}
M_{L(\log L)}f(x) \approx M^{2}f(x),
\end{equation}
where $M$ is the Hardy-Littlewood maximal function, and we will also
employ several times  the Kolmogorov inequality
\begin{equation}\label{kolmogorov}
\|f\|_{L^p(Q, \frac{dx}{|Q|})}\leq C\, \|f\|_{L^{q,\infty}(Q,
\frac{dx}{|Q|})},
\end{equation}
for $0<p<q<\infty$. See, e.g. \cite{W} and the reference in
\cite{LOPTT}.

\subsection{Sharp maximal functions}\label{sharp}

For $\delta>0$, $M_\delta$ is the maximal function
$$
M_\delta f(x)=M(|f|^\delta)^{1/\delta}(x)=\left (\sup_{Q\ni
x}\frac{1}{ |Q|}\int_Q |f(y)|^\delta \, dy \right )^{1/\delta}.
$$
In addition,  $M^\#$ is the sharp maximal function  of Fefferman and
Stein \cite{FS2},
$$M^\#(f)(x)=\sup_{Q\ni x}\inf_c \frac{1}{|Q|}\int_Q |f(y)-c|\, dy
\approx \sup_{Q\ni x}\frac{1}{|Q|}\int_Q |f(y)-f_Q|\, dy.
$$
%
and
$$
M^\#_\delta f(x)=M^\#(|f|^\delta)^{1/\delta}(x)
$$
We will also use from \cite{FS2}, the inequality
%
\begin{equation}\label{FSppstrong}
\int_{\mathbb R^n} (M_\d f(x))^p \, w(x)dx\leq C\,\int_{\mathbb R^n}
(M^{\#}_\d f(x))^p \, w(x)dx,
\end{equation}
for all function $f$ for which the left-hand side is finite,
and where  $0<p,\d<\infty$ and $w$ is a weight in $A_{\infty}$. 
%
 Moreover,
if  $\f :(0,\infty) \rightarrow (0,\infty)$ is doubling, then there
exists a constant $c$ (depending on the $A_{\infty}$ constant of $w$
and the doubling condition of $\f$) such that
\begin{equation} \label{FSweak}
\sup_{ \lambda >0} \f(\la) \, w( \{ y\in\RN: M_{\delta}f(y) < \la
\} ) \le c\, \sup_{ \la >0} \f(\la) \, w( \{ y\in \RN:
M^{\#}_{\delta}f(y) < \la  \} ),
\end{equation}
again for every function $f$ such that the left hand side is finite.
%


%

%


%

\subsection{Multiple weights}
Following the notation in \cite{LOPTT}, for $m$ exponents
$p_{1},\dots, p_{m}$, we will often write $p$ for the number given
by $\frac{1}{p}=\frac{1}{p_1}+\dots+\frac{1}{p_m}$, and $\vec P$ for
the vector $\vec P = (p_{1},\dots, p_{m})$.

Let  $1\le p_1, \dots , p_{m}<\infty$,  a multiple weight
$\vec w=(w_1,\dots,w_m)$,  is said to satisfy the {\it multilinear
$A_{\vec p}$ condition} if for
$$\nu_{\vec w}=\prod_{j=1}^mw_j^{p/p_j}.$$
it holds that
\begin{equation}\label{multiap}
\sup_{Q}\Big(\frac{1}{|Q|}\int_Q\nu_{\vec
w}\Big)^{1/p}\prod_{j=1}^m\Big(\frac{1}{|Q|}\int_Q
w_j^{1-p'_j}\Big)^{1/p'_j}<\infty.
\end{equation}
When $p_j=1$, $\Big(\frac{1}{|Q|}\int_Q w_j^{1-p'_j}\Big)^{1/p'_j}$
is understood as $\displaystyle(\inf_Qw_j)^{-1}$.

One can check that $A_{(1,\dots,1)}$ is contained in $A_{\vec P}$
for each $\vec P$, however the classes $A_{\vec P}$ are not
increasing with the natural partial order. As mentioned in the
introduction, these are the largest classes of weights for which the
multilinear \cz operators are bounded on Lebesgue spaces, as proved
in \cite{LOPTT}) improving on the results in \cite{GT5} and
\cite{PT}.  In fact,  one has
%
%
$$ \prod_{j=1}^m A_{p_j} \subset A_{\vec P},$$
with strict containment. Moreover, in general  $\vec w\in  A_{\vec
P}$ does not imply $w_j\in L^1_{\text{loc}}$ for any~$j$,
 but instead
%
%
%
%
%
%
%
%
%
\begin{equation}\label{cases}
\vec w \in A_{\vec P} \iff
\begin{cases}
w_j^{1-p'_j}\in A_{mp'_j},\,\,j=1,\dots,m
\\
\nu_{\vec w}\in A_{mp},
\end{cases}
\end{equation}
where the condition $w_j^{1-p'_j}\in A_{mp'_j}$ in the case $p_j=1$
is understood as $w_j^{1/m}\in A_1$.

Observe that in the linear case ($m=1$) both conditions included in
(\ref{cases}) represent the same $A_p$ condition. However, when
$m\ge 2$ neither of the conditions in (\ref{cases}) implies the
other. We refer the reader to \cite{LOPTT} for more details on this
multilinear weights.


\section{Proof of Theorem \ref{thmStrong}} \label{sectiontres}

The proof of the Theorem \ref{thmStrong} will rely on a pointwise estimate using sharp maximal functions.
The  technique of  comparing commutators with sharp maximal
operators has by now a long history of successful applications (see
the comments in \cite{LOPTT} p.15 and the references therein).

To state the pointwise result in great generality we need to introduce some additional for $m$-linear iterated commutators involving $j $ $BMO$ functions  with $j<m$.  Following  \cite{PTru},  
for positive integers $m$ and  $j$ with $1\leq j\leq m$,  we denote by $C_j^m$ the family of
all finite subsets $\sigma=\{\sigma(1),\dots,\sigma(j)\}$ of $\{1,\dots,m\}$ of $j$ different elements, where we always take $\sigma(k) < \sigma(l)$ if $k<l$. For any $\sigma \in C_j^m$, we associate the complementary sequence $\sigma' \in C^{m}_{m-j}$ given by $\sigma'=\{1,\dots,m\} \backslash\sigma$ with the convention 
 $C^{m}_{0}=\emptyset$. Given an $m$-tuple of functions {\rm {\bf b}} and  $\sigma \in C_j^m$, we also use the notation ${\rm {\bf b}}_{\sigma}$ for the $j$-tuple obtained from {\rm {\bf b}} given by $(b_{\sigma(1)}, \dots, b_{\sigma(j)})$. 

Similarly to \eqref{iterado}, we define for a \cz operator $T$,  $\sigma \in C_j^m$,  and ${\rm \bf b_\sigma} = (b_{\sigma(1)}, \dots, b_{\sigma(j)})$ in
$BMO^j$, the iterated commutator 
\begin{equation}\label{iterado2}
T_{\Pi\, {\rm {\bf b}_{\sigma}} }  \,(f_{1},\dots,f_{m})=
[b_{\sigma(1)}, [b_{\sigma(2)}, \dots [b_{\sigma({j-1})},[b_{\sigma(j)}, T]_{\sigma(j)}]_{\sigma({j-1})}\dots ]_{\sigma(2)}]_{\sigma(1)}({\rm \bf f}).      
\end{equation}
That is, formally 
\begin{equation*}
T_{\Pi\, {\rm {\bf b}_{\sigma}} } ({\rm {\bf f}})(x)=\int_{(\RN)^m} \left(\prod_{i=1}^j
(b_{\sigma(i)}(x)-b_{\sigma(i)}(y_{\sigma(i)}))\right) \, K(x,y_1,\dots, y_m)\prod_{i=1}^mf_i(y_i)\,
d\vec{y}.
\end{equation*}
Clearly $T_{\Pi\, {\rm {\bf b}_{\sigma}} }= \Tpi$ as defined before when $\sigma = \{1,2,\dots,m\}$, while 
 $T_{\Pi\, {\rm {\bf b}_{\sigma}} }= T^j_{b_j}$  when $\sigma =\{ j\}$.

The pointwise estimate that will serve our purposes  is the following.

\begin{theorem}\label{lemasharpconmuta}
Let $\Tpi$ be a multilinear commutator with
${\bf b}\in  BMO^m$ and  let
$0<\delta<\varepsilon$, with $0<\delta<1/m$. Then, there exists a
constant $C>0$, depending on $\delta$ and $\varepsilon$, such that
\begin{eqnarray}\label{lemasharpconmuta1}
M^{\#}_\delta (\Tpi\, ( {\rm \bf f}))(x)&\leq& C \prod_{j=1}^m\|
b_j\|_{BMO}\, \left ( \, {\mathcal M}_{L(\log L)}({\rm \bf f})(x)
+M_\varepsilon (T({\rm \bf f}))(x) \right )\\
&& + \quad C\sum_{j=1}^{m-1}\sum_{\sigma\in C_j^m}\prod_{i=1}^j\|b_{\sigma(i)}\|_{BMO}M_{\epsilon}(T_{\Pi\, {\rm {\bf b}_{\sigma'}} } ({\rm {\bf f}}))(x)\nonumber
\end{eqnarray}
for all $m$-tuples ${\rm \bf f}=(f_1,..,f_m)$ of bounded measurable
functions with compact support.
\end{theorem}

\begin{proof}
The way to interpret \eqref{lemasharpconmuta1} is 
$$M^{\#}_\delta (\Tpi\, ( {\rm \bf f}))(x) \lesssim \prod_{j=1}^m\|
b_j\|_{BMO}\, \, {\mathcal M}_{L(\log L)}({\rm \bf f})(x)
+ {\rm ``lower \, order \, terms"},
$$
as it will become apparent in its application.
Given the heavy technical notation and
for simplicity in the exposition,  we  only present the  case $m=2$.
As  the reader may soon see, the general case is 
only notationally more complicated and can be obtained with a similar procedure.
Hence, we will limit our selves to establish the following version of 
\eqref{lemasharpconmuta1}.  

For  $b_1,b_2\in BMO$ we will show that
\begin{eqnarray*}\label{miguala2}
M^{\#}_\delta (\Tpi\, ( f_1,f_2))(x)\! \!  &\leq&\! \! \!C\,  \|b_1\|_{BMO} \, \|b_2\|_{BMO} \, 
\left ( \, {\mathcal M}_{L(\log L)}(f_1,f_2)(x)
+M_\varepsilon (T(f_1,f_2))(x) \right )\\
&+& \! \! \! C \left(  \|b_2\|_{BMO} M_{\epsilon} (T_{b_1}^1  (f_1,f_2))(x) 
+ \|b_1\|_{BMO}M_{\epsilon}(T_{b_2}^2  (f_1,f_2))(x)
\right ). \nonumber
\end{eqnarray*}
For any constants $\lambda_1$ and $\lambda_2$, write 
\begin{eqnarray*}
\Tpi\, ({\rm \bf f})(x)\!\!&=&\!\!
(b_1(x)-\lambda_1)(b_2(x)-\lambda_2)T(f_1,f_2)(x)-(b_1(x)-\lambda_1)T(f_1,(b_2-\lambda_2)f_2)(x) \\
&&-(b_2(x)-\lambda_2)T((b_1-\lambda_1)f_1,f_2)(x)
+T((b_1-\lambda_1)f_1,(b_2-\lambda_2)f_2)(x).\nonumber\\
\!\!&=&\!\!-(b_1(x)-\lambda_1)(b_2(x)-\lambda_2)T(f_1,f_2)(x)+(b_1(x)-\lambda_1)T^2_{b_2-\lambda_2}(f_1,f_2)(x)\\
&&+(b_2(x)-\lambda_2)T^1_{b_1-\lambda_1}(f_1,f_2)(x)+ T((b_1-\lambda_1)f_1,(b_2-\lambda_2)f_2)(x).
\end{eqnarray*}

Also, if we fix $x\in \mathbb R^n$, a cube $Q$ centered at $x$ and
a constant $c$, then since $0<\delta<1/2$, we
can estimate
\begin{eqnarray*}
\left ( \frac{1}{|Q|}\int_Q \left | \Tpi ({\rm \bf f})(z)|^\delta
-|c|^{\delta} \right |\,dz\right )^{1/\delta} &\le&
\left ( \frac{1}{ |Q|}\int_Q \left | \Tpi({\rm \bf f})(z)
-c \right |^\delta
dz\right)^{1/\delta}\nonumber \\
\end{eqnarray*}

\begin{eqnarray*}
&\leq& \left ( \frac{C}{ |Q|}\int_Q  \left |(b_1(x)-\lambda_1)(b_2(x)-\lambda_2)T(f_1,f_2)(z) \right |^\delta dz\right )^{1/\delta}\nonumber \\
&&+ \left ( \frac{C}{ |Q|}\int_Q \left |
(b_1(x)-\lambda_1)T^2_{b_2-\lambda_2}(f_1,f_2)(z) \right
|^\delta dz\right )^{1/\delta}\nonumber \\
&&+ \left ( \frac{C}{ |Q|}\int_Q \left |
(b_2(x)-\lambda_2)T^1_{b_1-\lambda_1}(f_1,f_2)(z)\right
|^\delta dz\right )^{1/\delta}\nonumber\\
&&+ \left ( \frac{C}{ |Q|}\int_Q \left |
T((b_1-\lambda_1)f_1,(b_2-\lambda_2)f_2)(z)-c \right
|^\delta dz\right )^{1/\delta}\nonumber\\
&&=I+II+III+IV.
\end{eqnarray*}

We analyze each term separately selecting appropriate constants. Let
$Q^{*}=3Q$ and let $\lambda_j=(b_j)_{Q^{*}}$ be the average of $b_j$
on $Q^{*}$, $j=1,2$. For any $1<q_1,q_2,q_3<\infty$ with
$1=1/q_1+1/q_2+1/q_3$ and $q_3<\varepsilon/\delta$ we have by
H\"older's and Jensen's inequalities,
\begin{eqnarray*}
I &\leq &C\left ( \frac{1}{|Q|}\int_{Q}  \left | b_1(z)-\lambda_1
\right |^{\delta q_1} dz\right )^{1/\delta q_1}
\left ( \frac{1}{|Q|}\int_{Q} \left | b_2(z)-\lambda_2 \right |^{\delta q_2} dz\right )^{1/\delta q_2}\nonumber \\
&&\times\left ( \frac{1}{|Q|}\int_Q \left | T(f_1,f_2)(z) \right
|^{\delta q_3}dz\right )^{1/\delta
q_3}\nonumber \\
&\leq &C\|b_1\|_{BMO}\|b_2\|_{BMO}\,M_{\delta q_3}(T(f_1,f_2))(x)
\\[2mm]
&\leq
&C\|b_1\|_{BMO}\|b_2\|_{BMO}\,M_{\varepsilon}(T(f_1,f_2))(x),\nonumber
\end{eqnarray*}
which is an appropriate estimate for what we want to obtain.

Since  $II$ and $III$ are symmetric we only study $II$. Let $1<t_1,t_2<\infty$ with
$1=1/t_1+1/t_2$ and $t_2<\varepsilon/\delta$ then, by
H\"older's and Jensen's inequalities,
\begin{eqnarray*}
II&\leq&C\left ( \frac{1}{|Q|}\int_{Q}  \left | b_1(z)-\lambda_1
\right |^{\delta t_1} dz\right )^{1/\delta t_1}\left ( \frac{1}{|Q|}\int_Q \left | T^2_{b_2-\lambda_2}(f_1,f_2)(z) \right
|^{\delta t_2}dz\right )^{1/\delta
t_2}\nonumber \\
&\leq &C\|b_1\|_{BMO}\,M_{\delta t_2}(T^2_{b_2-\lambda_2}(f_1,f_2))(x)
\\
&\leq &C\|b_1\|_{BMO}\,M_{\varepsilon}(T^2_{b_2-\lambda_2}(f_1,f_2))(x)\\
 & = &   C\|b_1\|_{BMO}\,M_{\varepsilon}(T^2_{b_2}(f_1,f_2))(x).
\end{eqnarray*}
Similarly,
\begin{eqnarray*}
III \leq C\|b_2\|_{BMO}\,M_{\varepsilon}(T^1_{b_1-\lambda_1}(f_1,f_2))(x) = 
C\|b_2\|_{BMO}\,M_{\varepsilon}(T^1_{b_1}(f_1,f_2))(x).
\end{eqnarray*}

It only remain to study the last term $IV$. We split
each $f_i$ as $f_i=f_i^0+f_i^\infty$ where $f_i^0=f\chi_{Q^*}$ and
$f_i^\infty=f_i-f_i^0$. Let
$$
 c=\sum_{j=1}^3 c_{j},
$$
where
$$c_{1}=T(f_1^0,(b_2-\lambda_2)f_2^\infty)(x),$$
$$c_{2}=T(f_1^\infty,(b_2-\lambda_2)f_2^0)(x),$$
 $$c_{3}=T(f_1^\infty,(b_2-\lambda_2)f_2^\infty)(x).$$

Then,
\begin{eqnarray*}
IV&=&\left ( \frac{C}{ |Q|}\int_Q \left |
T((b_1-\lambda_1)f_1,(b_2-\lambda_2)f_2)(z)-c \right
|^\delta dz\right )^{1/\delta} \\[2mm]
&\leq &\left ( \frac{C}{ |Q|}\int_Q \left |
T((b_1-\lambda_1)f_1^0,(b_2-\lambda_2)f_2^0)(z) \right
|^\delta dz\right )^{1/\delta} \\[2mm]
&&+\left ( \frac{C}{ |Q|}\int_Q \left |
T((b_1-\lambda_1)f_1^0,(b_2-\lambda_2)f_2^\infty)(z)-c_{1} \right
|^\delta dz\right )^{1/\delta} \\[2mm]
&&+\left ( \frac{C}{ |Q|}\int_Q \left |
T((b_1-\lambda_1)f_1^\infty,(b_2-\lambda_2)f_2^0)(z)-c_{2} \right
|^\delta dz\right )^{1/\delta} \\[2mm]
&&+\left ( \frac{C}{ |Q|}\int_Q \left |
T((b_1-\lambda_1)f_1^\infty,(b_2-\lambda_2)f_2^\infty)(z)-c_{3}
\right
|^\delta dz\right )^{1/\delta} \\[2mm]
&=&IV_1+IV_2+IV_3+IV_4
\end{eqnarray*}

We choose $1<p<1/(2\delta)$. Since $p\delta<1/2$, we can estimate $IV_1$ using H\"{o}lder's inequality and the fact that $T$ is a Calder\'on-Zygmund operator
\begin{eqnarray*}
IV_1&\leq&\left ( \frac{C}{ |Q|}\int_Q \left |
T((b_1-\lambda_1)f_1^0,(b_2-\lambda_2)f_2^0)(z) \right
|^{p\delta} dz\right )^{1/(p\delta)} \\[2mm]
&\leq &
C\|T((b_1-\lambda_1)f_1^0,(b_2-\lambda_2)f_2^0)\|_{ L^{1/2,\infty}(Q, \frac{dx}{|Q|})}\nonumber \\[2mm]
&\leq &C \frac{1}{|Q|}\int_Q \left |(b_1(z)-\lambda_1)f_1^0(z)\right |dz\frac{1}{|Q|}\int_Q \left |(b_2(z)-\lambda_2)f_2^0(z)\right |dz\nonumber \\[2mm]
&\leq &C \|b_1\|_{BMO}\, \|f_1\|_{L(\log L),Q}\|b_2\|_{BMO}\, \|f_2\|_{L(\log L),Q} \nonumber \\[2mm]
&\leq &C\|b_1\|_{BMO} \|b_2\|_{BMO}\, {\mathcal M}_{L(\log
L)}(f_1,f_2)(x).
\end{eqnarray*}

Since $IV_2$ and $IV_3$ are symmetric, we consider for example $IV_2$, and estimate
\begin{eqnarray*}
&&|T((b_1-\lambda_1)f_1^0, (b_2-\lambda_2)f_2^\infty)(z)- T((b_1-\lambda_1)f_1^0, (b_2-\lambda_2)f_2^\infty)(x)|\\
&&\le \int_{3Q}|(b_1(y_1)-\lambda_1)f_1(y_1)|dy_1
\int\limits_{{\mathbb R}^n\setminus
3Q}\frac{|x-z|^{\e}|(b_2(y_2)-\lambda_2)f_2(y_2)|dy_2}
{(|z-y_1|+|z-y_2)^{2n+\e}}\\
&&\le \int_{3Q} |b_1(y_1)-\lambda_1)f_1(y_1)|dy_1
\sum_{k=1}^{\infty}\frac{|Q|^{\e/n}}{((3^k|Q|)^{1/n})^{2n+\e}}\int\limits_{(3^{k+1}Q)}
|(b_2(y_2)-\lambda_2)f_2(y_2)|dy_2\\
&&\le
C\sum_{k=1}^{\infty}\frac{|Q|^{\e/n}}{((3^k|Q|)^{1/n})^{2n+\e}}\left
(\int_{(3^{k+1}Q)}
|b_1(y_1)-\lambda_1)f_1(y_1)|\, dy_1\right) \times\\
&&\,\,\,\,\,\, \times \left(\int_{(3^{k+1}Q)}
|b_2(y_2)-\lambda_2)f_2(y_2)|\, dy_2\right)\\
&& \leq C\sum_{k=1}^{\infty}\frac{1}{3^{\e k}}\|b_1\|_{BMO}\|b_2\|_{BMO}\|f_1\|_{L(\log L),3^{k+1}Q}\|f_2\|_{L(\log L),3^{k+1}Q}\\[2mm]
&&\leq C\|b_1\|_{BMO} \|b_2\|_{BMO}\, {\mathcal M}_{L(\log
L)}(f_1,f_2)(x).
\end{eqnarray*}

\vspace{5mm} Fiinally, the term $IV_4$ is estimated in similar way
and we deduce
%
$$|T((b_1-\lambda_1)f_1^\infty, (b_2-\lambda_2)f_2^\infty)(z)- T((b_1-\lambda_1)f_1^\infty, (b_2-\lambda_2)f_2^\infty)(x)|\leq $$
$$\leq C\|b_1\|_{BMO} \|b_2\|_{BMO}\, {\mathcal M}_{L(\log L)}(f_1,f_2)(x).$$

The proof is complete.
\end{proof}

We note that we can also obtain analogous estimates to \eqref{lemasharpconmuta1}
for $m$-linear commutators involving $j<m$ functions in $BMO$. That is estimates of the form
\begin{equation}\label{lowerorder}
M^{\#}_\delta 
(T_{\Pi\, {\rm {\bf b}_{\sigma}} } ({\rm {\bf f}}))(x)
\lesssim 
\prod_{k=1}^j\|
b_{\sigma(k)}\|_{BMO}\,  {\mathcal M}_{L(\log L)_\sigma}({\rm \bf f})(x)
+ {\rm ``lower \, order \, terms"},
\end{equation}
where ${\mathcal M}_{L(\log L)_\sigma}$ denotes the analog of  
${\mathcal M}_{L(\log L)}$ but with only $\log$ factors in the ${\rm \bf f}_{\sigma}$ functions.
 (Note  that 
${\mathcal M}_{L(\log L)_\sigma} = {\mathcal M}^j_{L(\log L)}$ 
when $\sigma= \{j\}$.)
The lower order terms are now of the form
$$
\prod_{k=1}^l\|
b_{\eta'(k)}\|_{BMO}\, M_\epsilon (T_{\Pi\, {\rm {\bf b}_{\eta}} } ({\rm {\bf f}}))(x)
$$
for $l<j$, where $\eta$ is subset of  $\sigma$ of cardinality $l$, and $\eta \cup \eta' = \sigma$. Note also that 
trivially 
\begin{equation} \label{trivial}
{\mathcal M}_{L(\log L)_\sigma}({\rm \bf f})(x) \leq {\mathcal M}_{L(\log L)}({\rm \bf f})(x).
\end{equation}

These pointwise estimates are the key for the strong and
weak-type estimates with multiple weights.  In particular, they yield an appropriate version of the
following Coifman-Fefferman type inequalities (\cite{CF}).

\begin{theorem} \label{CommThm}
Let  $0<p<\infty$,  let $w$ be a weight in $A_\infty$, and suppose that
${\rm \bf b } \in  BMO^m$. Then, there exists a constant $C_w$
(independent of ${\rm \bf b }$) and a constant $c_w({\rm \bf b })$
such that
\begin{equation}\label{strongnorm}
\int_{\mathbb R^n} |\Tpi \,({\rm \bf f})(x)|^p w(x) dx \leq C_w\,
\prod_{j=1}^m \|b_j\|_{BMO} \, \int_{\mathbb R^n} {\mathcal
M}_{L(\log L)}({\rm \bf f})(x)^p w(x)dx,
\end{equation}
and
\begin{align}
\sup_{ t >0}\frac{1}{ \Phi^m (\frac{1}{t}) }w( \{ y\in\mathbb R ^{n}
: &
|\Tpi \,({\rm \bf f})(y)| >  t^m \} )\nonumber \\
 \le & \, c_w({\rm \bf b })\, \sup_{ t
>0} \frac{1}{ \Phi^m(\frac{1}{t}) } w( \{ y\in\mathbb R ^{n}: {\mathcal M}_{L(\log L)}({\rm \bf f})(y)> t^m \}),
\label{weaknorm}
\end{align}
for all ${\rm \bf f}=(f_1,..,f_m)$ bounded with compact support.
\end{theorem}

\begin{proof}
The proof of these types of estimates is by now standard. We refer
the reader to \cite[Theorem 1.6]{PTru} and \cite[Corollary 3.8 and
Theorem 3.19]{LOPTT}. The arguments there can be followed step by
step in this new case.  We briefly indicate such arguments in the case $m=2$, but,  as the reader will immediately notice, an iterative procedure using \eqref{lowerorder}  and \eqref{trivial} can be followed to obtain the general case.  

As already mentioned the philosophy in the approach is that terms involving $M_\epsilon$ can actually be treated as ``lower order terms". 
In fact, as in \cite{LOPTT} and using the Fefferman-Stein inequality \eqref{FSppstrong} , %
\begin{eqnarray}\label{tobecontinued}
\|  \Tpi(\vec f)\| _{L^p(w)}& \leq&\|M_\delta ( \Tpi(\vec
f))\|_{L^p(w)}
\nonumber \\
&\leq& C\,\|M^\#_\delta
(\Tpi(\vec f))\|_{L^p(w)}.
\end{eqnarray}
Using the pointwise estimate in the previous theorem and again the Fefferman-Stein inequality we can continue  from \eqref{tobecontinued} with
\begin{eqnarray*}
&\leq& C\, \|b_1\|_{ BMO}  \|b_2\|_{ BMO}\|  \left ( \|{\mathcal M}_{L(\log L)}(\vec
f)\|_{L^p(w)} +\|M_\varepsilon
(T(\vec f))\|_{L^p(w)}\right )\nonumber\\
&& + \,C \left(  \|b_2\|_{BMO} \| M_{\epsilon} (T_{b_1}^1  (f_1,f_2))\|_{L^p(w)}
+ \|b_1\|_{BMO} \|M_{\epsilon}(T_{b_2}^2  (f_1,f_2))\|_{L^p(w)}
\right )\\ \nonumber
&\leq& C\, \|b_1\|_{ BMO}  \|b_2\|_{ BMO}\| \left ( \|{\mathcal M}_{L(\log L)}(\vec f)\|_{L^p(w)}
+\|M^{\#}_\varepsilon
(T(\vec f))\|_{L^p(w)}\right )\nonumber\\
&& + \,C \left(  \|b_2\|_{BMO} \| M_{\epsilon} ^\#(T_{b_1}^1  (f_1,f_2))\|_{L^p(w)}
+ \|b_1\|_{BMO} \|M_{\epsilon}^\#(T_{b_2}^2  (f_1,f_2))\|_{L^p(w)}
\right ).\nonumber
\end{eqnarray*}
If we take $\epsilon$ small, we can now repeat the procedure using the results in \cite{LOPTT}
and estimate
$$\|M^{\#}_\varepsilon (T(\vec f))\|_{L^p(w)} \leq C  \|{\mathcal M}(\vec f)\|_{L^p(w)} 
 \leq C  \|{\mathcal M}_{L(\log L)}(\vec f)\|_{L^p(w)} ;
 $$
 and for $\epsilon < \epsilon'$,
 \begin{eqnarray*}
  \| M_{\epsilon} ^\#(T_{b_1}^1  (f_1,f_2))\|_{L^p(w)} &\leq & C
   \|b_1\|_{BMO} \left( \|{\mathcal M}_{L(\log L)}(\vec f)\|_{L^p(w)} + \|M_{\epsilon'}(T (f_1,f_2))\|_{L^p(w)}
  \right)\nonumber\\
  &\leq & C
   \|b_1\|_{BMO} \left( \|{\mathcal M}_{L(\log L)}(\vec f)\|_{L^p(w)} + \|M_{\epsilon'}^\#(T (f_1,f_2))\|_{L^p(w)}
  \right)\nonumber\\
&\leq&  C\,  \|b_1\|_{BMO}   \|{\mathcal M}_{L(\log L)}(\vec f)\|_{L^p(w)}.
\end{eqnarray*}
Similarly, 
$$
 \| M_{\epsilon} ^\#(T_{b_2}^2  (f_1,f_2))\|_{L^p(w)} \leq  C
  \|b_2\|_{BMO}   \|{\mathcal M}_{L(\log L)}(\vec f)\|_{L^p(w)}.
  $$
The desired inequality now follows.

 We observe that to use the
Fefferman-Stein inequality as argued in
\cite[pp.32-33]{LOPTT},  one needs to verify that  certain terms in the left-hand side 
of the inequalities are finite when the right-hand side ones are (when the right-hand are infinite there is nothing to prove). However, if one assumes ${\rm \bf b }$ in
$(L^\infty)^m$, then everything  is clear because of the boundedness properties
of $T$. The passage to ${\rm \bf b }$ in $BMO^m$ is standard,
and combining it with Fatou's lemma, one gets the desire result.

The proof of \eqref{weaknorm} also follows the pattern for the
corresponding estimate relating $\Tsigma$ and ${\mathcal
M}_{\Sigma\,L(\log L)}$ in \cite[pp. 33-35]{LOPTT}.  
 We also briefly indicate some of the details needed when $m=2$.  
To further simplify the presentation, and since we do not intend to keep track of the exact dependence on ${\rm \bf b }$, we assume that the $BMO$ norms of the functions $b_js$ are equal to one.
Note that  the doubling properties of $\Phi$ will produce a constant 
$c({\rm \bf b})$  in the general case.  It should be noted though that,  unlike the strong case, such constant  is not multilinear in ${\rm \bf b }$.

Using the pointwise estimate for $M^\#_\delta (\Tpi(\vec f))$ we get 
$$
\sup_{t>0}\frac{1}{\Phi^2 (\frac{1}{t}) }w( \{ y\in\mathbb R ^{n}
|\Tpi \,({\rm \bf f})(y)| >  t^2 \} ) \leq  \sup_{t>0}\frac{1}{\Phi^2 (\frac{1}{t}) }w( \{ y\in\mathbb R ^{n}
M_{\delta}(\Tpi \,({\rm \bf f})(y)) >  t^2 \} )
$$
\begin{eqnarray*}
&\leq&\sup_{t>0}\frac{1}{\Phi^2 (\frac{1}{t}) }w\left(\left\{M^\#_\delta
(\Tpi(\vec f))(x)>t^2 \right\} \right)\\
&\leq& C\, \sup_{t>0}\frac{1}{\Phi^2 (\frac{1}{t}) } w\left(\left\{
{\mathcal M}_{L(\log L)}(f_1,f_2)(x)>t^2\right\}\right)\\
&&+C\sup_{t>0}\frac{1}{\Phi^2 (\frac{1}{t}) }w\left(\left\{ \,M_\varepsilon (T(f_1,f_2))(x)>t^2\ \right\}\right)\\
&&+ C  \sup_{t>0}\frac{1}{\Phi^2 (\frac{1}{t}) }w\left( \left\{M_{\epsilon} (T_{b_1}^1  (f_1,f_2))(x)>t^2\right\}\right)\\
&&+ C \sup_{t>0}\frac{1}{\Phi^2 (\frac{1}{t}) }w\left(\left\{M_{\epsilon}(T_{b_2}^2  (f_1,f_2))(x)>t^2\right\}\right)\\
&=& I+II+III+IV.
\end{eqnarray*}
We claim that the main term is  $I$, which will give the desired result. In fact, from the estimates
\begin{equation}\label{uno}
M^\#_\varepsilon (T(f_1,f_2))(x)\leq {\mathcal M}(\vec f)(x)\leq {\mathcal M}_{L\log L}(\vec f)(x)
\end{equation}
and the weak-type version of the Fefferman-Stein inequality  \eqref{FSweak} we easily get that  $II\lesssim I$.

To estimate $III$ we invoke again   \eqref{FSweak} and the results for  the $M^\#$  function of commutators of lower order to get  for $\epsilon < \epsilon'$,
 \begin{eqnarray*}
  III&\leq& C  \sup_{t>0}\frac{1}{\Phi^2 (\frac{1}{t}) }w\left( \left\{M^{\#}_{\epsilon} (T_{b_1}^1  (f_1,f_2))(x)>t^2\right\}\right)\\
  &\leq&\sup_{t>0}\frac{1}{\Phi^2 (\frac{1}{t}) }w\left( \left\{{\mathcal M}_{L(\log L)}(\vec f)(x)>t^2\right\}\right)\\
  &&+\sup_{t>0}\frac{1}{\Phi^2 (\frac{1}{t}) }w\left( \left\{M_{\epsilon'}(T (f_1,f_2))(x)>t^2\right\}\right).
  \end{eqnarray*}
 Iterating the procedure and using (\ref{uno}) we  arrive to $III\lesssim I$. The term $IV$ is completely analogous.  Again, to be able to apply \eqref{FSweak}  some justification  is needed. But one can always assume the weight to be bounded and use a limiting process. We refer to \cite{LOPTT} and omit the rest of the details.

\end{proof}

{\it Proof of Theorem 1.1}
We can now easily finish the proof of Theorem~\ref{thmStrong}. Since for
$\vec w$ in $A_{\vec P}$, the weight $\nu_{\vec w}$ is in
$A_\infty$, we can use one more result from \cite{LOPTT} on strong
bounds for ${\mathcal M}_{L(\log L)}$ and conclude from
\eqref{strongnorm} that
$$
\int_{\mathbb R^n} |\Tpi \,({\rm \bf f})(x)|^p \nu_{\vec w}(x) dx
\leq C_{\nu_{\vec w}}\, \prod_{j=1}^m \|b_j\|_{BMO} \, \int_{\mathbb
R^n} ({\mathcal M}_{L(\log L)}({\rm \bf f})(x))^p \nu_{\vec w}(x)dx
$$
$$
\leq C_{\nu_{\vec w}}\, \prod_{j=1}^m \|b_j\|_{BMO} \,\prod_{j=1}^m
\|f_j\|_{L^{p_j}(w_j)} .
$$
\section{Proof of Theorem \ref{commutatorthmllogl}}\label{seccioncuatro}

We start with a new weak type end-point estimate for ${\mathcal
M}_{L(\log L)}$ .

\begin{theorem} \label{maximalthmllogl}
Let $\vec w \in A_{\vec 1}$. Then there exists a constant $C$ such
that
\begin{equation}\label{debilmaximal}
 \nu_{\vec w} \left( \big \{x \in \RN\,:\, {\mathcal M}_{L(\log L)}({\rm \bf f})(x)|>t^m \big \}  \right)
\leq C\, \prod_{j=1}^m \left (\int_{\RN} \Phi^{(m)}\left(
\dfrac{|f_j(x)|}{t} \right)\,\, w_j(x)dx\right )^{1/m}.
\end{equation}
 Morever, this estimate is sharp in the sense that $\Phi^{(m)}$ can not be
replaced by $\Phi^{(k)}$ for $k<m$.
\end{theorem}

\begin{proof}

Our goal is to estimate $|\Omega|=|\{{\mathcal M}_{L(\log
L)}(f_1,f_2,....,f_m)>1\}|$. The set $\Omega$ is open and we may
assume it to be not empty.
 It is enough then to control the size of every compact set $F$
contained in $\Omega$.

 For
$x\in F$ there exists a cube $Q$ with $x\in Q$ such that

\begin{equation}\label{cubo}
\prod_{i=1}^m \|f_{i}\|_{\Phi,Q}>1.
\end{equation}
Thus, by a covering argument,  we can extract a finite family of
disjoint cubes $\{Q_{i}\}$ whose dilations cover $F$ for which
\begin{equation}\label{uno1}
| F|\leq C\sum_{i}|Q_i|
\end{equation}
and $\{Q_i\}$ satisfies
\begin{equation*}
\prod_{i=1}^m \|f_{i}\|_{\Phi,Q_i}>1.
\end{equation*}

We use again the  notation  $C_h^m$ for the
family of all subset $\sigma=(\sigma(1),...,\sigma(h))$ extracted from  the
set of indeces $\{1,...,m\}$ using $1\leq h\leq m$ different elements.  Given
$\sigma\in C_h^m$ and a cube $Q_i$, we say that $i\in B_{\sigma}$ if
$\|f_{\sigma(k)}\|_{\Phi,Q_i}>1$ for $k=1,...,h$ and
$\|f_{\sigma(k)}\|_{\Phi,Q_i}\leq1$ for $k=h+1,...,m$.

Let us  consider  $\sigma\in C_h^m$ and $i\in B_{\sigma}$. Denote
$$\Pi_k=\prod_{j=1}^k\|f_{\sigma(j)}\|_{\Phi, Q_i}$$
and $\Pi_0=1$. Then it is easy to check that $\Pi_k>1$ for every
$1\leq k\leq m$. It follows that
\begin{equation*}
1<\Pi_{k}=\|f_{\sigma(k)}\|_{\Phi, Q_i}\,
\Pi_{k-1}=\|f_{\sigma(k)}\Pi_{k-1}\|_{\Phi, Q_i}
\end{equation*}
or, equivalently  (by \eqref{biggerone})
\begin{equation}\label{12}
\frac{1}{|Q_i|}\int_{Q_i} \Phi\left(f_{\sigma(k)}\,
\Pi_{k-1}\right)>1.
\end{equation}
In particular,
\begin{equation}\label{13}
1<\frac{1}{|Q_i|}\int_{Q_i} \Phi\left(f_{\sigma(m)}\,
\Pi_{m-1}\right)\leq\frac{1}{|Q_i|}\int_{Q_i}
\Phi\left(f_{\sigma(m)}\right) \Phi\left(\Pi_{m-1}\right).
\end{equation}
Now, by taking into account the following equivalence
\begin{equation*}
\|f\|_{\Phi, Q}\simeq \inf_{\mu>0}\{\mu+\frac{\mu}{|Q|}\int_Q
\Phi(|f|/\mu)\},
\end{equation*}
if $1\leq j\leq m-h-1$, by $(\ref{12})$ we get
\begin{eqnarray*}
\Phi^j(\Pi_{m-j})&=&\Phi^j(\|f_{\sigma(m-j)}\Pi_{m-j-1}\|_{\Phi,
Q_i})\\
&\leq&C\Phi^j\left(1+\frac{1}{|Q_i|}\int_{Q_i}
\Phi\left(f_{\sigma(m-j)}\, \Pi_{m-j-1}\right)\right)\\
&\leq&C\frac{1}{|Q_i|}\int_{Q_i}
\Phi^{j+1}\left(f_{\sigma(m-j)}\right)\,
\Phi^{j+1}\left(\Pi_{m-j-1}\right).
\end{eqnarray*}
From $(\ref{13})$, by iterating the inequality above, we obtain
\begin{eqnarray*}
1&<&C\frac{1}{|Q_i|}\int_{Q_i}
\Phi\left(f_{\sigma(m)}\right)\frac{1}{|Q_i|}\int_{Q_i}
\Phi^2\left(f_{\sigma(m-1)}\right)\Phi^2\left(\Pi_{m-2}\right)\\
&\leq&C\left(\prod_{j=0}^{m-h-1}\frac{1}{|Q_i|}\int_{Q_i}
\Phi^{j+1}\left(f_{\sigma(m-j)}\right)\right)\Phi^{m-h}\left(\Pi_{h}\right)\\
&\leq&C
\left(\prod_{j=0}^{m-h-1}\frac{1}{|Q_i|}\int_{Q_i}\Phi^{j+1}\left(f_{\sigma(m-j)}\right)\right)
\left(\prod_{j=1}^h\Phi^{m-h}(\|f_{\sigma(j)}\|_{\Phi, Q_i})\right).
\end{eqnarray*}
since $\Phi$ is submultiplicative.

Thus, since $i\in B_{\sigma}$, we have  $\|f_{\sigma(j)}\|_{\Phi,
Q_i}>1$ for $j=1,...h$, and it follows
\begin{equation}\label{cuatro}
1<C\left(\prod_{j=0}^{m-h-1}\frac{1}{|Q_i|}\int_{Q_i}\Phi^{j+1}\left(f_{\sigma(m-j)}\right)\right)
\left(\prod_{j=1}^h\frac{1}{|Q_i|}\int_{Q_i}\Phi^{m-h+1}(f_{\sigma(j)})\right).
\end{equation}

Now, since for $1\leq h\leq m$ and $0\leq j\leq m-h-1$ we have that
$\Phi^{j+1}(t)\leq \Phi^{m-h}(t)\leq \Phi^m(t)$ and
$\Phi^{m-h+1}(t)\leq \Phi^m(t)$, we deduce
\begin{equation*}
1<C\prod_{j=1}^m\frac{1}{|Q_i|}\int_{Q_i}\Phi^{m}(f_{j})
\end{equation*}
or equivalently
\begin{equation*}
|Q_i|<C\prod_{j=1}^m\left(\int_{Q_i}\Phi^{m}(f_j)\right)^{1/m}.
\end{equation*}
Thus, going back to $(\ref{uno1})$ it follows that
\begin{eqnarray*}
\nu_{\vec w} \big(F \big)^m  & \approx & \left(\sum_i \nu_{\vec
w}(Q_i)\right)^m\\
 &\le &\left (\sum_{h=1}^{m}\sum_{\sigma\in
C_{h}^{m}}\sum_{i\in B_{\sigma}}\nu_{\vec
w}(Q_i)\right )^m\\
&\leq&C\left (\sum_{h=1}^{m}\sum_{\sigma\in C_{h}^{m}}\sum_{i\in
B_{\sigma}}\prod_{j=1}^m  \inf_{Q_i} w_j^{1/m}
|Q_i|^{1/m} \left ( \frac{1}{|Q_i|}\int_{Q_i}\Phi^{m}(f_{j})\right)^{1/m}\right )^m\\
&\leq&C\prod_{j=1}^m   \left (\int_{\mathbb R^n} \Phi^{m}( f_j(y))
w_j(y)\, dy \right)
\end{eqnarray*}
which concludes the proof of \eqref{debilmaximal}

We now prove that the estimate \eqref{debilmaximal} is sharp in the
sense stated in theorem.

We claim that the following estimate is false
\begin{equation}\label{impossible}
|\{ x: {\mathcal M}_{L(\log L)}({\rm \bf f }) > \la^{m}\} | \leq C\,
 \left(
\prod_{j=1}^m \|
\Phi^{m-1}(\frac{|f_{j}|}{\la})\|_{L^{1}}\right)^{1/m}
\end{equation}
 We let
$\la=1$ and then the estimate to be studied is
\begin{equation}
|\{ x: {\mathcal M}_{L(\log L)}({\rm \bf f }) >1 \} |^m \leq
C\,\prod_{j=1}^m \| \Phi^{m-1}({|f_{j}|})\|_{L^{1}}
\end{equation}
for any ${\rm \bf f }$ with all the components positives. Hence by
the same homogeneity we replacing $f_1$ by \,$\frac{f_1}{\la^m}$
\begin{equation}\label{impossiblem=2}
|\{ x: {\mathcal M}_{L(\log L)}({\rm \bf f }) > \la^m\} |^m \leq C\,
{\int_{\R} \Phi^{m-1}\left(\frac{f_1}{\la^m}\right)\prod_{j=2}^m
\int_{\R} \Phi^{m-1}\left(f_j\right)}
\end{equation}

Now, let $f_j=\chi_{(0,1)}$. If \eqref{impossiblem=2} holds, since
$\Phi$ is a Young function, we conclude
\begin{equation}\label{debilconmutacontraejemplo}
\sup_{\la>0} \frac{1}{\Phi{^{m-1}}(\la^{-m})} |\{x\in
\R\,:\,{\mathcal M}_{L(\log L)}({\rm \bf f })(x)|>\la^m\}|^m\leq C.
\end{equation}
However, observe that, by definition of ${\mathcal M}_{L(\log
L)}({\rm \bf f })$ and of \,$\|.\|_{L(\log L),Q}$, it follows for
any subset $A$ that
$\|\chi_A\|_{\Phi,Q}=\frac{1}{\Phi^{-1}(\frac{|Q|}{|A\cap Q|})}$.
Hence, if  $x>e$ we have
$${\mathcal M}_{L(\log L)}({\rm \bf f })(x) \geq \|\chi_{(0,1)}\|_{L(\log L),(0,x)}^m=
\frac{1}{\Phi^{-1}(x)^m}  $$
Thus, {taking into account that $\phi^k(t)\cong t(1+\log^+ t)^k$ },
the left-hand side of \eqref{debilconmutacontraejemplo} is bigger
than
\begin{eqnarray*}
 \sup_{\la>0}
\frac{1}{\Phi{^{m-1}}(\la^{-m})} |\{x > e:\,
\frac{1}{\Phi^{-1}(x)}>\la\}|^m &\geq &\sup_{0<\la<{1/e}} \frac{
\left(\Phi(\frac{1}{\la})-e\right)^{{m}}
}{\Phi{^{m-1}}(\frac{1}{\la^{{m}}}) }\\
&\geq& \frac1{{2^m}}\,\sup_{0<\la<\frac{1}{2e}} \frac{
\left(\Phi(\frac{1}{\la})\right)^{{m}} }{\Phi{^{m-1}}(\frac{1}{\la^{{m}}}) }\\
&{\geq}&{\frac{C}{m^{m-1}2^{m}}}\,\sup_{0<\la<\frac{1}{2e}} \log \frac{1}{\la}\\
&=&\infty
\end{eqnarray*}

\end{proof}

Given \eqref{weaknorm} and \eqref{debilmaximal} the proof of
Theorem~\ref{commutatorthmllogl} is almost routine. The reader can
see \cite[pp.38-39]{LOPTT} and easily adapt the arguments.


\end{document}